\documentclass[11pt]{amsart}
\usepackage[T1]{fontenc}
\usepackage{lmodern}
\usepackage{amsmath,amssymb,amsthm,mathtools}
\usepackage[margin=1in]{geometry}
\usepackage{microtype}
\usepackage{xcolor}
\usepackage[colorlinks=true,allcolors=blue]{hyperref}

\numberwithin{equation}{section}
\newtheorem{theorem}{Theorem}[section]
\newtheorem{proposition}[theorem]{Proposition}
\newtheorem{lemma}[theorem]{Lemma}

\newtheorem*{theoremA}{Theorem A}
\newtheorem*{theoremB}{Theorem B}
\newtheorem*{corollaryC}{Corollary C}
\theoremstyle{definition}
\newtheorem{definition}[theorem]{Definition}
\theoremstyle{remark}
\newtheorem{remark}[theorem]{Remark}

\title[Positive sectional curvature on exotic \(8\)- and \(10\)-spheres]{Positive sectional curvature on the exotic \(8\)-sphere and order-three homotopy \(10\)-spheres}
\author{Fernando Galaz-Garc\'ia}
\address{Department of Mathematical Sciences, Durham University, Durham, United Kingdom}
\email{fernando.galaz-garcia@durham.ac.uk}
\date{\today}
\begin{document}
\begin{abstract}
We prove that the exotic smooth \(8\)-sphere and both oriented homotopy \(10\)-spheres representing elements of order three admit Riemannian metrics with strictly positive sectional curvature. The construction combines Sperança's special \(S^3\)-\(S^3\) bundle models with the compatible-disk construction of He, Liu and Yau. We show that the relevant bundles admit equivariant polar normal forms, with transition functions that are constant along meridians and conjugation-equivariant. We also show that the representation-dependent part of the He--Liu--Yau construction requires only a uniform bound on the infinitesimal action fields. Sperança's \(8\)- and \(10\)-dimensional examples satisfy this bound. The resulting northern and southern metrics have matching boundary metrics and compatible second fundamental forms, so the Reiser--Wraith gluing theorem gives the required positively curved metrics.
\end{abstract}
\maketitle

\section{Introduction and results}
The construction of Riemannian metrics with positive sectional curvature is a longstanding problem in global Riemannian geometry. Beyond the compact rank-one symmetric spaces, the known examples have historically been sparse and have arisen from a small number of constructions involving homogeneous spaces, biquotients, fibrations, and manifolds with large symmetry; see the surveys of Grove \cite{GroSurvey} and Ziller \cite{ZiSurvey}. Exotic spheres have played a particularly important role in this problem. Gromoll and Meyer constructed a metric of nonnegative sectional curvature on an exotic \(7\)-sphere \cite{GM}, Grove and Ziller obtained nonnegative curvature on the Milnor spheres \cite{GZ}, and Goette, Kerin and Shankar later showed that every exotic \(7\)-sphere admits nonnegative sectional curvature \cite{GKS}. Wilhelm constructed a metric with positive sectional curvature almost everywhere on the Gromoll--Meyer sphere \cite{Wi}, and Eschenburg and Kerin subsequently obtained an almost positively curved metric on the same sphere \cite{EK}. Petersen and Wilhelm proposed a metric with strictly positive sectional curvature on the Gromoll--Meyer sphere \cite{PW}. More recently, Ouyang \cite{OuGM} and Guo, Fang and Lu \cite{GFLGM} have proposed two further constructions of positively curved metrics on the Gromoll--Meyer sphere. For Ricci curvature, Wraith proved that every homotopy sphere bounding a parallelisable manifold admits a metric of positive Ricci curvature \cite{WrRicci}; see also the survey of Joachim and Wraith \cite{JW} for the broader curvature picture on exotic spheres.

Recently, there has also been renewed activity in the construction of metrics with positive sectional curvature on other manifolds. Among several new preprints, Brendle and Hung propose such a metric on \(S^2\times S^2\) \cite{BH}, Guo, Fang and Lu propose such metrics on \(S^2\times S^3\) and \(S^3\times S^3\) \cite{GFL23,GFL33}, Deng, Hu and Zhang propose positively curved metrics on a family of homotopy \(11\)-spheres \cite{DHZ}, and He, Liu and Yau propose positively curved metrics on all smooth homotopy \(7\)-spheres \cite{HLY}. The latter construction realises each smooth homotopy \(7\)-sphere as the gluing of two compatible positively curved quotient disks. On the southern side they use a connection metric, while on the northern side they use a doubly warped metric. The metric profiles are chosen so that the induced boundary metrics agree and the sum of the outward second fundamental forms is positive. The two disks can then be smoothly glued while preserving positive sectional curvature by a gluing theorem of Reiser and Wraith \cite{RW}. We show that the same construction applies to the exotic homotopy spheres in dimensions eight and ten arising from Sperança's special \(S^3\)-\(S^3\) bundles \cite{Sp}. These are the first even-dimensional exotic homotopy spheres known to admit even non-negative sectional curvature.

Let \(\Theta_m\) denote the group of oriented diffeomorphism classes of smooth homotopy \(m\)-spheres. Kervaire--Milnor \cite[Theorem~1.1 and p.~504]{KM} showed that
\[
\Theta_8\cong\mathbb Z/2,
\qquad
\Theta_{10}\cong\mathbb Z/6.
\]
The nontrivial element of \(\Theta_8\) is represented by the unique exotic \(8\)-sphere up to oriented diffeomorphism.

\begin{theoremA}
The exotic smooth \(8\)-sphere admits a smooth Riemannian metric with strictly positive sectional curvature.
\end{theoremA}

\begin{theoremB}
Both oriented homotopy \(10\)-spheres representing elements of order three in \(\Theta_{10}\) admit smooth Riemannian metrics with strictly positive sectional curvature.
\end{theoremB}

\begin{corollaryC}
Every smooth homotopy \(8\)-sphere admits positive sectional curvature. A smooth homotopy \(10\)-sphere admits positive sectional curvature if and only if it admits positive scalar curvature.
\end{corollaryC}

The proofs rest on two main observations. The first is bundle-theoretic: the special \(S^3\)-\(S^3\) bundles arising in Sperança's construction admit equivariant polar normal forms. In suitable northern and southern trivialisations, the transition function is constant along meridians and conjugation-equivariant. The second observation is geometric and concerns the \(S^3\)-representation on the equatorial sphere. The southern connection-metric argument of He--Liu--Yau is already formulated in arbitrary dimension once its hypotheses are verified. Their northern construction, however, is written for a particular orthogonal \(S^3\)-representation. For that representation they obtain the estimate
\[
\|K_y\|_{\mathrm{op}}\le2
\]
for the infinitesimal action map. We show that orthogonality together with this bound is the only representation-specific input in the northern horizontal curvature estimates and the boundary comparison. This gives a positive-curvature construction for any polar bundle satisfying the same estimate, and Sperança's bundles satisfy precisely this bound.

We note that the representation bound above is specific to the He--Liu--Yau construction used in this paper, rather than an intrinsic restriction on the special \(S^3\)-\(S^3\) bundles considered here. Deng, Hu and Zhang \cite[Proposition~5.1]{DHZ} have obtained a compatible-disk theorem for an arbitrary finite-dimensional orthogonal \(S^3\)-representation and an arbitrary smooth conjugation-equivariant clutching map. Combined with the polar normal form and quotient description developed here, and with the gluing theorem of Reiser and Wraith \cite[Theorem~A(i)]{RW}, these results imply that every quotient arising from the special \(S^3\)-\(S^3\) bundles considered in this paper admits a metric of positive sectional curvature; see Remark~\ref{rem:general_bound}. This observation leads naturally to the study of arbitrary orthogonal \(S^3\)-representations for the constructions we consider and of the topology of the resulting quotients.

The paper is organised as follows. Sections~\ref{sec:star} and~\ref{sec:polartriv} develop the equivariant bundle theory and establish the polar normal form. Section~\ref{sec:HLY} extends the He--Liu--Yau construction to polar bundles satisfying
\[
\|K\|_{\mathrm{op}}\le2.
\]
In Section~\ref{sec:speranca}, we identify Sperança's \(8\)- and \(10\)-dimensional examples as star bundles, verify the required action-field bound, and apply the resulting transfer theorem. Corollary~C then follows from Theorems~A and~B together with the \(\alpha\)-invariant obstruction to positive scalar curvature.

\subsection*{Acknowledgments and AI Assistance}
Generative AI tools (OpenAI ChatGPT Astra and Anthropic Claude Opus 5.5) were
used under the author's direction to assist with mathematical exploration,
drafting, literature searches, and internal review of this work. The author
independently checked the mathematical arguments and references in the final
manuscript, made all final decisions concerning its contents, and takes full
responsibility for the paper. The author would like to thank Martin Kerin for helpful comments and suggestions.

\section{Star bundles, polar bundles, and their quotients}\label{sec:star}

In this section, we introduce the equivariant bundle structures used throughout the paper, motivated by the constructions of Sperança and He--Liu--Yau \cite{Sp,HLY}. Starting from Sperança's special \(S^3\)-\(S^3\) bundles, we incorporate the orthogonal action on the base sphere and its fixed axis into the notion of a \emph{star bundle}. We then introduce polar bundles, whose transition functions are constant along meridians, and describe their star quotients as two disks glued by an explicit attaching map. In the next section we show that every star bundle admits such a polar normal form.

We begin with Sperança's notion of a special \(G\)-\(G\) bundle \cite[\S1.1]{Sp}, specialised to \(G=S^3\) and with the principal action written on the right.

\begin{definition}[Special \(S^3\)-\(S^3\) bundle]
A \emph{special \(S^3\)-\(S^3\) bundle} is a smooth principal right \(S^3\)-bundle
\[
\pi\colon E\to M
\]
together with a smooth free left \(S^3\)-action \(\star\) on \(E\) which commutes with the principal right action and covers an \(S^3\)-action on \(M\). Thus
\[
q\star(ph)=(q\star p)h
\qquad
(q,h\in S^3,\ p\in E),
\]
and there is an \(S^3\)-action on \(M\) such that
\[
\pi(q\star p)=q\cdot\pi(p).
\]
\end{definition}

\begin{definition}[Star bundle]\label{def:starbundle}
A \emph{star bundle} consists of:
\begin{enumerate}
\item a Euclidean vector space
\[
W=\mathbb Re\oplus V,
\qquad \dim W=n+1,\quad n\ge3,
\]
where \(e\) is a distinguished unit vector and \(V=e^\perp\);
\item an orthogonal representation
\[
\rho\colon S^3\to \mathrm{O}(W)
\]
fixing \(e\), and hence preserving \(V\);
\item a special \(S^3\)-\(S^3\) bundle
\[
\pi\colon E\to S^n=S(W)
\]
whose induced \(S^3\)-action on the base is the restriction of \(\rho\). Equivalently,
\[
\pi(q\star p)=\rho(q)\pi(p)
\qquad
(q\in S^3,\ p\in E).
\]
\end{enumerate}
\end{definition}

We use geodesic polar coordinates around \(e\). Every \(\zeta\in S^n\setminus\{o_N,o_S\}\) can be written uniquely as
\[
\zeta=(\cos t)e+(\sin t)x,
\qquad 0<t<\pi,\quad x\in S(V).
\]
The endpoints \(t=0\) and \(t=\pi\) correspond to the poles
\[
o_N=e,\qquad o_S=-e,
\]
where the angular variable \(x\) is not defined. Set
\[
\mathcal U_N=S^n\setminus\{o_S\},
\qquad
\mathcal U_S=S^n\setminus\{o_N\}.
\]
Since \(\rho\) fixes \(e\) and preserves \(V\), it preserves the radial coordinate \(t\) and acts on the angular variable through \(\rho|_V\). We now introduce the polar model that will serve as a normal form for star bundles.

\begin{definition}[Polar bundle]\label{def:polar}
Let
\[
\theta\colon S(V)\to S^3
\]
be a smooth map satisfying
\begin{equation}\label{eq:equiv}
\theta(\rho(q)x)=q\,\theta(x)\,q^{-1}
\qquad
(q\in S^3,\ x\in S(V)).
\end{equation}
The \emph{polar bundle} \(P_\theta\to S^n\) is the principal right \(S^3\)-bundle obtained by gluing
\[
\mathcal U_N\times S^3
\qquad\text{and}\qquad
\mathcal U_S\times S^3
\]
over \(\mathcal U_N\cap\mathcal U_S\) via
\begin{equation}\label{eq:transition}
u_S=\theta(x)u_N.
\end{equation}
Thus the transition function depends only on the angular variable \(x\), and is constant along each meridian.

The principal action is
\[
(\zeta,u)h=(\zeta,uh),
\]
and the star action is defined in both trivialising charts by
\begin{equation}\label{eq:staraction}
q\star(\zeta,u)
=
(\rho(q)\zeta,qu).
\end{equation}
\end{definition}

The equivariance condition \eqref{eq:equiv} ensures that the star action \eqref{eq:staraction} is compatible with the transition function \eqref{eq:transition}. Indeed, if \(u_S=\theta(x)u_N\), then
\[
\theta(\rho(q)x)\,qu_N
=q\theta(x)u_N
=qu_S.
\]
Thus the star action is globally well defined. This is the conjugation-equivariance mechanism used by Sperança in \cite[(3.1)--(3.2)]{Sp}, where it guarantees that the corresponding local \(S^3\)-actions glue to a global action on the total space. The star action is smooth, including at the poles, because its local formula involves only the smooth linear action \(\rho\) on the base and left multiplication on \(S^3\), and does not involve the angular variable \(x\). The star action is free, since \(qu=u\) implies \(q=1\), and it commutes with the principal right action. The same compatibility mechanism also appears in \cite[\S2.4]{HLY}.

Let \(\pi_\theta:P_\theta\to S^n\) denote the bundle projection, and fix \(0<a<\pi\). Set
\[
U_N=\{t\le a\},
\qquad
U_S=\{t\ge a\}.
\]
These are \(S^3\)-invariant closed geodesic balls centred at \(o_N\) and \(o_S\), of radii
\[
R_N=a,\qquad R_S=\pi-a.
\]
Put
\[
P_\alpha=\pi_\theta^{-1}(U_\alpha),
\qquad
D_\alpha=P_\alpha/S^3_\star,
\qquad \alpha\in\{N,S\}.
\]
We identify \(U_\alpha\) with the closed disk of radius \(R_\alpha\) in \(V\) by the geodesic-polar coordinates
\[
\zeta_N=t\,x,
\qquad
\zeta_S=(\pi-t)\,x.
\]

The following lemma gives the quotient-disk description for general star bundle data. It generalises the explicit construction in \cite[\S2.5]{HLY} and is an instance of Sperança's reentrance mechanism \cite[\S3.1]{Sp}.

\begin{lemma}[Quotient disks and attaching map]\label{lem:attaching}
\leavevmode
\begin{enumerate}
\item In the product chart of \(P_\alpha\), the map
\[
\Psi_\alpha(\zeta_\alpha,u_\alpha)
=
\rho(u_\alpha)^{-1}\zeta_\alpha
\]
is constant on star orbits. It induces a diffeomorphism of \(D_\alpha\) onto the closed disk of radius \(R_\alpha\) in \(V\), with inverse
\[
Y\longmapsto[(Y,1)]_\star.
\]

\item With boundary markings
\[
y_\alpha=\rho(u_\alpha)^{-1}x\in S(V),
\]
the identification
\[
\partial D_N\longrightarrow\partial D_S
\]
induced by \eqref{eq:transition} is
\begin{equation}\label{eq:sigma}
y_S=\sigma(y_N),
\qquad
\sigma(y)=\rho(\theta(y))^{-1}y.
\end{equation}
The map \(\sigma\) is a diffeomorphism of \(S(V)\) with inverse
\[
\widehat\sigma(y)=\rho(\theta(y))y.
\]

\item
\[
P_\theta/S^3_\star
\cong
 D_N^n\cup_\sigma D_S^n.
\]
\end{enumerate}
\end{lemma}

\begin{proof}
For \(\alpha\in\{N,S\}\), consider
\[
\Psi_\alpha(\zeta_\alpha,u_\alpha)
=
\rho(u_\alpha)^{-1}\zeta_\alpha.
\]
This map is constant on star orbits, since
\[
\Psi_\alpha(\rho(q)\zeta_\alpha,qu_\alpha)
=
\rho(qu_\alpha)^{-1}\rho(q)\zeta_\alpha
=
\rho(u_\alpha)^{-1}\zeta_\alpha.
\]
Moreover, every star orbit contains a unique point with fibre coordinate \(1\). Indeed,
\[
u_\alpha^{-1}\star(\zeta_\alpha,u_\alpha)
=
(\rho(u_\alpha)^{-1}\zeta_\alpha,1).
\]
Thus the slice
\[
 D^n_{R_\alpha}\times\{1\}
\]
meets every star orbit exactly once, and \(\Psi_\alpha\) descends to a diffeomorphism
\[
D_\alpha
\longrightarrow
D^n_{R_\alpha},
\]
with inverse
\[
Y\longmapsto[(Y,1)]_\star.
\]
This proves (1).

We now determine the identification of the two boundary spheres. Let \(x\in S(V)\) be the angular coordinate of a point on the common boundary \(t=a\), and write
\[
y_\alpha=\rho(u_\alpha)^{-1}x.
\]
The transition relation is
\[
u_S=\theta(x)u_N.
\]
Since \(x=\rho(u_N)y_N\), the equivariance of \(\theta\) gives
\[
\theta(x)
=
\theta(\rho(u_N)y_N)
=
u_N\theta(y_N)u_N^{-1}.
\]
Therefore
\[
\begin{aligned}
y_S
&=\rho(u_S)^{-1}x\\
&=\rho(\theta(x)u_N)^{-1}x\\
&=\rho\!\left(u_N^{-1}\theta(x)^{-1}u_N\right)y_N\\
&=\rho(\theta(y_N))^{-1}y_N.
\end{aligned}
\]
Thus the attaching map is
\[
\sigma(y)=\rho(\theta(y))^{-1}y.
\]

To see that \(\sigma\) is a diffeomorphism, define
\[
\widehat\sigma(y)=\rho(\theta(y))y.
\]
Using the equivariance of \(\theta\) with \(q=\theta(y)^{-1}\) and \(q=\theta(y)\), respectively, we obtain
\[
\theta(\sigma(y))=\theta(y),
\qquad
\theta(\widehat\sigma(y))=\theta(y).
\]
Consequently,
\[
\widehat\sigma(\sigma(y))
=
\rho(\theta(y))\rho(\theta(y))^{-1}y
=
y,
\]
and similarly
\[
\sigma(\widehat\sigma(y))=y.
\]
Hence
\[
\widehat\sigma=\sigma^{-1},
\]
which proves (2).

Finally,
\[
P_\theta=P_N\cup P_S,
\]
and, by (1), the quotients \(D_N\) and \(D_S\) are closed \(n\)-disks. Since the transition function
\[
u_S=\theta(x)u_N
\]
is independent of the radial variable \(t\), the induced identification on a product collar of the common boundary is constant in the collar direction and has angular part \(\sigma\). Thus the quotient is the smooth disk gluing
\[
P_\theta/S^3_\star
\cong
D_N^n\cup_\sigma D_S^n.
\]
This proves (3).
\end{proof}

\section{Equivariant polar trivialisations}\label{sec:polartriv}

The polar bundles introduced in the previous section provide the objects needed for the metric construction. We now show that every star bundle admits equivariant northern and southern trivialisations whose transition function is constant along meridians. The construction uses an \(S^3\)-invariant connection and parallel transport along meridians from the two fixed poles. In particular, every star bundle admits, up to equivariant bundle isomorphism, a polar representation $P_\theta$ for a smooth conjugation-equivariant map $\theta\colon S(V)\to S^3$. This provides the bridge from Sperança's special \(S^3\)-\(S^3\) bundles to the polar framework used in the He--Liu--Yau metric construction.

\begin{proposition}\label{prop:polar}
Let
\[
\pi\colon E\to S^n=S(W)
\]
be a star bundle. Then there exist a smooth map
\[
\theta\colon S(V)\to S^3
\]
satisfying
\[
\theta(\rho(q)x)=q\,\theta(x)\,q^{-1},
\]
and a principal-bundle isomorphism
\[
\Phi\colon P_\theta\longrightarrow E
\]
which is equivariant with respect to the star actions. Consequently,
\[
E/S^3_\star\cong P_\theta/S^3_\star.
\]
\end{proposition}

\begin{proof}

\emph{Step 1: an invariant connection.}
Choose a smooth principal connection \(\omega_0\) on \(E\). For \(q\in S^3\), let
\[
L_q(p)=q\star p.
\]
Since the star action commutes with the principal right action and covers \(\rho\), the map \(L_q\) is a principal-bundle automorphism. Hence, \(L_q^*\omega_0\) is again a principal connection.

Let \(dq\) denote normalised Haar measure on \(S^3\), and define
\[
\omega=\int_{S^3}L_q^*\omega_0\,dq.
\]
The defining properties of a principal connection are linear in the connection form, so \(\omega\) is again a principal connection. Moreover, for every \(a\in S^3\),
\[
L_a^*\omega
=
\int_{S^3}L_{qa}^*\omega_0\,dq
=
\omega
\]
by invariance of Haar measure. Thus \(\omega\) is invariant under the star action.

\emph{Step 2: normalising the star action over the poles.}
Since \(\rho\) fixes \(o_N=e\), the fibre \(E_{o_N}\) is preserved by the star action. Choose \(p_0\in E_{o_N}\). Since the principal right action is simply transitive on \(E_{o_N}\), for every \(q\in S^3\) there is a unique element \(\varphi(q)\in S^3\) such that
\[
q\star p_0=p_0\varphi(q).
\]
The map \(\varphi\colon S^3\to S^3\) is smooth, since both the star orbit map \(q\mapsto q\star p_0\) and the fibre identification \(h\mapsto p_0h\) are smooth.

Since the star and principal actions commute, \(\varphi\) is a homomorphism:
\[
\begin{aligned}
(q_1q_2)\star p_0
&=q_1\star(q_2\star p_0)\\
&=q_1\star(p_0\varphi(q_2))\\
&=(q_1\star p_0)\varphi(q_2)\\
&=p_0\varphi(q_1)\varphi(q_2).
\end{aligned}
\]
Hence
\[
\varphi(q_1q_2)=\varphi(q_1)\varphi(q_2).
\]

The map \(\varphi\) is injective because the star action is free. Since \(\varphi\) is a Lie group homomorphism,
\[
\varphi(\exp(tX))
=
\exp\bigl(t\,d\varphi_1(X)\bigr)
\qquad
(X\in\operatorname{Im}\mathbb H).
\]
If \(d\varphi_1(X)=0\), then
\[
\varphi(\exp(tX))=1
\]
for every \(t\). By injectivity of \(\varphi\), the one-parameter subgroup \(\exp(tX)\) is trivial, and hence \(X=0\). Thus \(d\varphi_1\) is injective. Since the domain and codomain have the same dimension, \(d\varphi_1\) is an isomorphism. It follows from the inverse function theorem that the image of \(\varphi\) contains a neighbourhood of the identity, and hence is an open subgroup of \(S^3\). It is also compact, hence closed. Since \(S^3\) is connected, \(\varphi\) is surjective and therefore an automorphism.

Every automorphism of \(S^3\cong \mathrm{SU}(2)\) is inner \cite[Lemma~5.1]{BJ}. Thus there exists \(c\in S^3\) such that
\[
\varphi(q)=cqc^{-1}.
\]
Set
\[
s_0=p_0c.
\]
Then
\begin{equation}\label{eq:model}
q\star s_0
=
(q\star p_0)c
=
p_0cqc^{-1}c
=
s_0q
\qquad(q\in S^3).
\end{equation}
Since \(\rho\) also fixes \(o_S=-e\), the same argument gives a point \(s_0'\in E_{o_S}\) satisfying
\[
q\star s_0'=s_0'q
\qquad(q\in S^3).
\]

\emph{Step 3: sections by parallel transport along meridians.}
For \(x\in S(V)\), let
\[
c_x(t)=(\cos t)e+(\sin t)x,
\qquad 0\le t\le\pi,
\]
be the unit-speed meridian from \(o_N\) to \(o_S\) in the direction \(x\). For \(0\le t<\pi\), define \(s_N(c_x(t))\) by parallel transporting \(s_0\) along \(c_x|_{[0,t]}\) with respect to the connection \(\omega\).

Equivalently, since
\[
\exp_{o_N}:B_\pi(0)\subset V\longrightarrow\mathcal U_N
\]
is a diffeomorphism, we may write
\[
s_N(\exp_{o_N}v)=\tau_v(s_0),
\]
where \(\tau_v\) denotes parallel transport along
\[
\tau\longmapsto\exp_{o_N}(\tau v),
\qquad 0\le\tau\le1.
\]
Parallel transport depends smoothly on \(v\). Hence, \(s_N\) is a smooth section over \(\mathcal U_N\).

Similarly,
\[
\exp_{o_S}:B_\pi(0)\subset V\longrightarrow\mathcal U_S
\]
is a diffeomorphism. Parallel transporting \(s_0'\) along the radial geodesics from \(o_S\) defines a smooth section
\[
s_S\colon\mathcal U_S\to E.
\]

\emph{Step 4: equivariance of the radial sections.}
Since \(\rho(q)\) is orthogonal and fixes \(o_N=e\),
\[
\rho(q)\circ c_x=c_{\rho(q)x}.
\]
Moreover, \(L_q(p)=q\star p\) preserves the connection \(\omega\), and therefore maps horizontal lifts to horizontal lifts. The principal right action also preserves the horizontal distribution.

Thus
\[
t\longmapsto q\star s_N(c_x(t))
\]
is a horizontal lift of \(c_{\rho(q)x}\) beginning at
\[
q\star s_0=s_0q.
\]
On the other hand,
\[
t\longmapsto s_N(c_{\rho(q)x}(t))\,q
\]
is a horizontal lift of the same curve with the same initial point. By uniqueness of horizontal lifts,
\[
q\star s_N(c_x(t))
=
s_N(c_{\rho(q)x}(t))\,q.
\]
Equivalently,
\begin{equation}\label{eq:sectionequiv}
s_N(\rho(q)\zeta)
=
(q\star s_N(\zeta))q^{-1},
\qquad
\zeta\in\mathcal U_N.
\end{equation}
The same argument gives the corresponding identity for \(s_S\).

\emph{Step 5: polar trivialisations.}
Define
\[
\Phi_N\colon\mathcal U_N\times S^3\longrightarrow \pi^{-1}(\mathcal U_N),
\qquad
\Phi_N(\zeta,u)=s_N(\zeta)u.
\]
Since \(s_N\) is a smooth section, \(\Phi_N\) is a smooth principal-bundle trivialisation. By \eqref{eq:sectionequiv} and the commutation of the star and principal actions,
\[
\begin{aligned}
q\star\Phi_N(\zeta,u)
&=(q\star s_N(\zeta))u\\
&=s_N(\rho(q)\zeta)\,qu\\
&=\Phi_N(\rho(q)\zeta,qu).
\end{aligned}
\]
Thus, in this trivialisation, the star action has the polar form
\[
q\star(\zeta,u)=(\rho(q)\zeta,qu).
\]
The same construction gives a star-equivariant principal-bundle trivialisation
\[
\Phi_S\colon\mathcal U_S\times S^3\longrightarrow \pi^{-1}(\mathcal U_S).
\]

\emph{Step 6: the polar transition function.}
On \(\mathcal U_N\cap\mathcal U_S\), define the transition function \(\theta\) by
\[
s_N(\zeta)=s_S(\zeta)\theta(\zeta).
\]
Fix \(x\in S(V)\). Both \(s_N\circ c_x\) and \(s_S\circ c_x\) are horizontal lifts of the meridian \(c_x\). Choose \(t_0\in(0,\pi)\) and let \(g\in S^3\) be determined by
\[
s_N(c_x(t_0))=s_S(c_x(t_0))g.
\]
Since right translation preserves the horizontal distribution, \(s_S\circ c_x\,g\) is a horizontal lift of \(c_x\) with the same value as \(s_N\circ c_x\) at \(t_0\). Uniqueness of horizontal lifts therefore gives
\[
s_N(c_x(t))=s_S(c_x(t))g
\]
for all \(0<t<\pi\). Hence \(\theta(c_x(t))\) is independent of \(t\); we write it simply as \(\theta(x)\). Since the transition function is smooth on the overlap, this defines a smooth map
\[
\theta\colon S(V)\to S^3.
\]

The identity \(s_N=s_S\theta\) implies
\[
\Phi_N(\zeta,u_N)
=
\Phi_S(\zeta,\theta(x)u_N),
\]
so the transition between the two trivialisations is
\[
u_S=\theta(x)u_N.
\]

Finally, using \eqref{eq:sectionequiv} for both sections,
\[
\begin{aligned}
s_S(\rho(q)\zeta)\theta(\rho(q)x)
&=s_N(\rho(q)\zeta)\\
&=(q\star s_N(\zeta))q^{-1}\\
&=(q\star s_S(\zeta))\theta(x)q^{-1}\\
&=s_S(\rho(q)\zeta)\,q\theta(x)q^{-1}.
\end{aligned}
\]
Therefore
\[
\theta(\rho(q)x)=q\theta(x)q^{-1}.
\]
Thus \(\theta\) is conjugation-equivariant, and \(\Phi_N,\Phi_S\) assemble to a star-equivariant principal-bundle isomorphism
\[
\Phi\colon P_\theta\longrightarrow E.
\]
\end{proof}

The averaging argument and the construction of equivariant local trivialisations have precedents in Sperança's treatment of special \(G\)-\(G\) bundles \cite[Lemmas~4.3 and~4.5]{Sp}. Proposition~\ref{prop:polar} produces the northern and southern trivialisations simultaneously from the two fixed poles, so that their transition function is constant along each meridian, as required for the metric construction in the next section.

\section{The He--Liu--Yau construction for star bundles with
\texorpdfstring{\(\|K\|_{\mathrm{op}}\le2\)}{|K|<=2}}
\label{sec:HLY}

We now turn to the geometric part of the construction, building on the compatible-disk construction of He--Liu--Yau \cite{HLY}. In this section, we build compatible metrics on the two quotient disks and control their sectional curvature and boundary geometry. We follow the curvature convention of He--Liu--Yau and write
\[
\mathcal K(X\wedge Y)
\]
for the sectional-curvature numerator. For an orthonormal pair \(X,Y\), this is the sectional curvature of the plane they span. 

For a Riemannian manifold with smooth boundary, we take the second fundamental form with respect to the outward unit normal \(\nu\):
\[
\operatorname{II}(v,w)=\langle\nabla_v\nu,w\rangle.
\]
With this convention, the boundary of a sufficiently small geodesic ball has positive definite second fundamental form; see \cite[\S3.2.1]{Pet}. In the construction below, we write
\[
\mathcal B_N,\mathcal B_S
\]
for the second fundamental forms of the source boundaries \(\partial P_N\) and \(\partial P_S\), equipped with \(G_N\) and \(G_S\), and
\[
B_N,B_S
\]
for the second fundamental forms of the quotient boundaries \(\partial D_N\) and \(\partial D_S\), equipped with \(g_N\) and \(g_S\).

Throughout this section, \(S^3\subset\mathbb H\) is the group of unit quaternions, \(Q\) is the Euclidean inner product on \(\operatorname{Im}\mathbb H\), and \(h_{S^3}\) is the corresponding bi-invariant round metric of sectional curvature one.

Fix \(0<a<\pi\). As in Section~\ref{sec:star}, let
\[
U_N=\{t\le a\},
\qquad
U_S=\{t\ge a\},
\]
be the northern and southern geodesic disks in \(S^n\), and set
\[
P_N=\pi_\theta^{-1}(U_N),
\qquad
P_S=\pi_\theta^{-1}(U_S).
\]
Their star quotients
\[
D_N=P_N/S^3_\star,
\qquad
D_S=P_S/S^3_\star
\]
are \(n\)-disks, and Lemma~\ref{lem:attaching} identifies the full quotient as
\[
P_\theta/S^3_\star
\cong
D_N\cup_\sigma D_S.
\]

He--Liu--Yau \cite{HLY} construct positively curved metrics on these two quotient disks by equipping \(P_N\) and \(P_S\) with different \(S^3\)-invariant metrics and passing to the star quotients. On the southern piece they use a connection metric over the round disk \(U_S\). On the northern piece they use a doubly warped metric, which in our notation has the form
\[
ds^2+F(s)^2h_{S^{n-1}}+r(s)^2h_{S^3}.
\]
The metric profiles are chosen so that the quotient metrics agree along their common boundary and the sum of their outward second fundamental forms is positive. The two quotient disks are then smoothly glued, while preserving positive sectional curvature, by the gluing theorem of Reiser--Wraith \cite[Theorem~A(i)]{RW}.

The northern construction in \cite{HLY} is carried out for a particular orthogonal \(S^3\)-representation. For that representation, He--Liu--Yau obtain the estimate
\[
\|K_y\|_{\mathrm{op}}\le2
\]
for the corresponding infinitesimal action map. Our observation is that this estimate is the only representation-specific input in their horizontal curvature estimates and boundary comparison. Their argument therefore extends to any polar bundle for which the same bound holds.

For \(y\in S(V)\), define the infinitesimal action map
\[
K_y\colon \operatorname{Im}\mathbb H\longrightarrow T_yS(V)
\]
by
\[
K_y\xi
=
\left.\frac{d}{d\tau}\right|_{\tau=0}
\rho(e^{\tau\xi})y.
\]
Thus \(K_y\xi\) is the value at \(y\) of the infinitesimal vector field generated by \(\xi\). We equip \(\operatorname{Im}\mathbb H\) with the Euclidean inner product \(Q\) and \(T_yS(V)\) with the unit round metric, and write
\[
\|K_y\|_{\mathrm{op}}
=
\sup_{|\xi|=1}|K_y\xi|.
\]

\begin{theorem}\label{thm:HLYgeneral}
Let
\[
P_\theta\to S^n,
\qquad n\ge3,
\]
be a polar bundle. Suppose that
\[
\|K_y\|_{\mathrm{op}}\le2
\qquad
\text{for every }y\in S(V).
\]
Then the star quotient
\[
P_\theta/S^3_\star
\]
admits a smooth Riemannian metric with strictly positive sectional curvature.
\end{theorem}

\subsection{Positive connection metrics}

We first recall the result of He--Liu--Yau that supplies the southern filling.

\begin{proposition}[\protect{\cite[Proposition~3.1]{HLY}}]
\label{fact:prop31}
Let
\[
\pi\colon P\to(B,g_B)
\]
be a principal right \(S^3\)-bundle over a compact base, possibly with boundary, with
\[
\operatorname{sec}_{g_B}\ge\kappa>0.
\]
Let \(\omega\) be a connection with curvature \(\Omega\), and set
\[
M_0=\sup_B|\Omega|,
\qquad
M_1=\sup_B|D\Omega|,
\]
where \(D\Omega\) is defined using the Levi--Civita connection of \(g_B\) and the induced connection on the adjoint bundle. With the norms
\[
|\Omega|^2
=
\sum_{i<j,a}(\Omega_{ij}^a)^2,
\qquad
|D\Omega|^2
=
\sum_{k,i<j,a}((D_k\Omega)_{ij}^a)^2,
\]
set \(C=4\). Suppose that
\begin{equation}\label{eq:Lambda}
\Lambda
\ge
16CM_0^2
+
64C^2\kappa^{-1}(M_1+1)^2
+
4
\end{equation}
and that a smooth function \(\varphi\) on \(B\) satisfies
\[
-\operatorname{Hess}\varphi
\ge
\Lambda g_B.
\]
Then, for all sufficiently small \(\varepsilon>0\), the connection metric
\[
G_\varepsilon
=
\pi^*g_B+r_\varepsilon^2Q(\omega,\omega),
\qquad
r_\varepsilon^2
=
\varepsilon e^{\varepsilon\varphi},
\]
has strictly positive sectional curvature.
\end{proposition}

The proposition is formulated for arbitrary base dimension. Its proof gives explicit smallness conditions on \(\varepsilon\), equations (3.18)--(3.20) of \cite{HLY}, and the uniform estimate
\[
\operatorname{sec}_{G_\varepsilon}
\ge
\min\left\{
\frac{\kappa}{2},
\frac{\varepsilon\Lambda}{8},
\frac{1}{8\varepsilon}
\right\}>0;
\]
see \cite[(3.22)]{HLY}.

\subsection{The southern filling}

The southern filling follows the construction of He--Liu--Yau \cite[\S4]{HLY}. We adapt it to the polar bundle \(P_\theta\), using the connection-metric proposition recalled above.

Fix
\[
\frac{\pi}{2}<a<b_0<\pi
\]
and choose a smooth function
\[
\chi\colon[a,\pi]\to[0,1]
\]
such that \(\chi=0\) near \(a\) and \(\chi=1\) on \([b_0,\pi]\). The southern base
\[
U_S=\{a\le t\le\pi\}\subset S^n
\]
carries the round metric
\[
g_B=dt^2+\sin^2t\,h_{S^{n-1}}
\]
with
\[
\operatorname{sec}_{g_B}=1.
\]

We first construct a connection on \(P_\theta|_{U_S}\). Regard \(\theta\) as a function of the angular variable \(x\) alone. In the northern and southern trivialisations define the local connection forms
\begin{equation}\label{eq:potentials}
A_N=\chi\,\theta^{-1}d\theta,
\qquad
A_S=(\chi-1)\,d\theta\,\theta^{-1}.
\end{equation}
In either trivialisation the corresponding principal connection is written
\[
\omega=u^{-1}Au+u^{-1}du.
\]

These local forms agree on the overlap. From
\[
u_S=\theta u_N
\]
one obtains
\[
u_S^{-1}du_S
=
u_N^{-1}\theta^{-1}d\theta\,u_N
+
u_N^{-1}du_N
\]
and
\[
u_S^{-1}A_Su_S
=
u_N^{-1}\theta^{-1}A_S\theta\,u_N.
\]
Thus the two expressions for \(\omega\) agree precisely when
\[
A_N
=
\theta^{-1}A_S\theta+\theta^{-1}d\theta.
\]
Since
\[
\theta^{-1}A_S\theta
=
(\chi-1)\theta^{-1}d\theta,
\]
this identity holds. Hence \(A_N\) and \(A_S\) define a global connection \(\omega\).

Near the boundary \(t=a\), one has \(A_N=0\), so \(\omega\) is the product connection in the northern trivialisation. Since \(\chi=1\) on \([b_0,\pi]\), the local connection form \(A_S\) vanishes on a neighbourhood of the south pole \(o_S\). Thus
\[
\omega=u^{-1}du
\]
there in the southern trivialisation, and the connection extends smoothly across \(o_S\).

The connection is invariant under the star action. The conjugation-equivariance of \(\theta\) implies
\[
\rho(q)^*(\theta^{-1}d\theta)
=
q(\theta^{-1}d\theta)q^{-1},
\]
and similarly
\[
\rho(q)^*(d\theta\,\theta^{-1})
=
q(d\theta\,\theta^{-1})q^{-1}.
\]
Since \(\rho\) preserves \(t\), it also preserves \(\chi(t)\). Hence the local connection forms transform by conjugation. Together with \(u\mapsto qu\), this gives
\[
(qu)^{-1}(qAq^{-1})(qu)+(qu)^{-1}d(qu)
=
u^{-1}Au+u^{-1}du.
\]
Thus \(\omega\) is star-invariant. The round metric on \(U_S\) is also star-invariant because \(\rho\) is orthogonal.

Let
\[
\vartheta=\theta^{-1}d\theta.
\]
In the northern trivialisation,
\[
A_N=\chi\vartheta,
\]
and the Maurer--Cartan equation
\[
d\vartheta+\vartheta\wedge\vartheta=0
\]
gives
\[
\Omega_N
=
d\chi\wedge\vartheta
+
\chi(\chi-1)\vartheta\wedge\vartheta.
\]
The connection \(\omega\) is smooth on the compact base \(U_S\). Consequently, its curvature and covariant derivative are bounded:
\[
M_0=\sup_{U_S}|\Omega|<\infty,
\qquad
M_1=\sup_{U_S}|D\Omega|<\infty.
\]

We now choose the function required by Proposition~\ref{fact:prop31}. Set
\[
\varphi(t)=-A_0\cos t.
\]
On the round sphere,
\[
\operatorname{Hess}(\cos t)=-\cos t\,g_B,
\]
and therefore
\[
-\operatorname{Hess}\varphi
=
-A_0\cos t\,g_B.
\]
Since \(t\ge a>\pi/2\),
\[
-\cos t\ge|\cos a|,
\]
so
\[
-\operatorname{Hess}\varphi
\ge
A_0|\cos a|\,g_B.
\]
Choose \(\Lambda\) satisfying \eqref{eq:Lambda} with \(\kappa=1\), and then choose \(A_0>0\) so that
\[
A_0|\cos a|\ge\Lambda.
\]
Proposition~\ref{fact:prop31} then gives \(\varepsilon_S>0\) such that, for
\[
0<\varepsilon<\varepsilon_S,
\]
the connection metric
\begin{equation}\label{eq:south}
G_S
=
dt^2+\sin^2t\,h_{S^{n-1}}
+r_S(t)^2Q(\omega,\omega),
\qquad
r_S^2
=
\varepsilon e^{-\varepsilon A_0\cos t},
\end{equation}
has strictly positive sectional curvature.

The metric extends smoothly across \(o_S\). With
\[
\tau=\pi-t,
\]
one has
\[
r_S(\pi-\tau)
=
\sqrt{\varepsilon}\,
e^{\varepsilon A_0\cos\tau/2},
\]
which is a smooth even function of \(\tau\). Since \(A_S=0\) near \(\tau=0\), the connection is the product connection there. Hence \(G_S\) extends smoothly across the south pole. The star action is free and isometric, so the quotient map
\[
(P_S,G_S)\longrightarrow(D_S,g_S)
\]
is a Riemannian submersion. O'Neill's formula therefore gives
\[
\operatorname{sec}_{g_S}>0
\]
throughout \(D_S\), including its centre and boundary.

Near \(t=a\), the connection is the product connection in the northern trivialisation, so the induced metric on the source boundary \(\partial P_S\) is
\[
F_a^2h_{S^{n-1}}+r_a^2h_{S^3},
\]
where
\begin{equation}\label{eq:bdata}
F_a=\sin a,
\qquad
r_a^2=\varepsilon e^{\varepsilon A_0|\cos a|},
\qquad
q_s=\frac12\varepsilon A_0F_a,
\qquad
\mu_S=\cot a<0.
\end{equation}
Here
\[
q_s=\frac{r_S'(a)}{r_a},
\qquad
\mu_S=\frac{(\sin t)'|_{t=a}}{\sin a}.
\]
We write \(X\) and \(U\) for the angular and fibre components, respectively, measured in the corresponding scaled orthonormal frames. Since the outward unit normal to \(U_S\) along \(t=a\) is \(-\partial_t\), the second fundamental form of the source boundary is
\[
\mathcal B_S(X+U,X+U)
=
-\mu_S|X|^2-q_s|U|^2.
\]
These are the boundary data to be matched by the northern filling.

\subsection{The northern filling}

The northern filling follows the construction of He--Liu--Yau \cite[\S5]{HLY}. Their choice of warping functions will be used unchanged. We show that the curvature argument extends from their particular \(S^3\)-representation to the general polar action under the hypothesis
\[
\|K_y\|_{\mathrm{op}}\le2.
\]

In the northern trivialisation,
\[
P_N\cong \mathbb D^n\times S^3,
\]
and the connection is the product connection \(u^{-1}du\). It is invariant under the star action since
\[
(qu)^{-1}d(qu)=u^{-1}du.
\]
On \(P_N\) consider the \(S^3\)-invariant metric
\begin{equation}\label{eq:north}
G_N
=
ds^2+F(s)^2h_{S^{n-1}}+r(s)^2h_{S^3},
\qquad
0\le s\le\ell_N.
\end{equation}
The underlying disk is identified with the original northern polar disk by the radial diffeomorphism
\[
t=\frac{a}{\ell_N}s,
\]
which preserves the angular variable \(x\).

\paragraph{The star-horizontal distribution.}
We use scaled orthonormal components for the two round factors in \eqref{eq:north}. At a point with fibre coordinate \(u=1\), the infinitesimal star vector generated by
\[
\xi\in\operatorname{Im}\mathbb H
\]
has angular and fibre components
\[
(FK_y\xi,r\xi).
\]
A vector
\[
\lambda\partial_s+X+U
\]
is therefore star-horizontal precisely when
\[
\langle X,FK_y\xi\rangle+\langle U,r\xi\rangle=0
\qquad
\text{for every }\xi\in\operatorname{Im}\mathbb H.
\]
Equivalently,
\[
U=TX,
\qquad
T=-\frac{F}{r}K_y^*.
\]
Since
\[
\|K_y^*\|_{\mathrm{op}}
=
\|K_y\|_{\mathrm{op}}
\le2,
\]
we obtain
\begin{equation}\label{eq:Tbound}
\|T\|_{\mathrm{op}}\le\frac{2F}{r}.
\end{equation}
This is the representation estimate used in the northern curvature calculation.

Consider two star-horizontal vectors
\[
\lambda\partial_s+X+U,
\qquad
\mu\partial_s+Y+V,
\]
where
\[
U=TX,\qquad V=TY.
\]
Then \eqref{eq:Tbound} gives
\begin{equation}\label{eq:graphestimates}
|\lambda V-\mu U|^2
\le
\frac{4F^2}{r^2}
|\lambda Y-\mu X|^2
\end{equation}
and
\begin{equation}\label{eq:areaestimate}
|X\otimes V-Y\otimes U|^2
\le
\frac{8F^2}{r^2}|X\wedge Y|^2.
\end{equation}
The first inequality follows immediately from
\[
\lambda V-\mu U
=
T(\lambda Y-\mu X).
\]
For the second, if \(X\) and \(Y\) are independent, make a determinant-one change of the pair so that they are orthogonal. Both alternating expressions in \eqref{eq:areaestimate} are unchanged, while
\[
|X\otimes TY-Y\otimes TX|^2
=
|X|^2|TY|^2+|Y|^2|TX|^2.
\]
The estimate then follows from \eqref{eq:Tbound}. If \(X\) and \(Y\) are dependent, both sides of \eqref{eq:areaestimate} vanish. These are the analogues of \cite[(5.4)--(5.6)]{HLY}. Principal right translation gives the same graph description at every point of the fibre.

\paragraph{Curvature of star-horizontal planes.}
For two vectors as above, the sectional-curvature numerator of the doubly warped metric \eqref{eq:north} is
\begin{align}\label{eq:fullcurvature}
\mathcal K={}&
-\frac{F''}{F}|\lambda Y-\mu X|^2
-\frac{r''}{r}|\lambda V-\mu U|^2
\notag\\
&+
\frac{1-F'^2}{F^2}|X\wedge Y|^2
+
\frac{1-r'^2}{r^2}|U\wedge V|^2
\notag\\
&-
\frac{F'r'}{Fr}
|X\otimes V-Y\otimes U|^2.
\end{align}
This is the dimension-independent form of \cite[(5.7)]{HLY}. It follows from the standard curvature formulas for a doubly warped product. In radially parallel scaled orthonormal frames \(E_i\) and \(E_a\) on the two round factors,
\[
\nabla_{E_i}\partial_s=\frac{F'}F E_i,
\qquad
\nabla_{E_a}\partial_s=\frac{r'}r E_a,
\qquad
\nabla_{E_i}E_a=\nabla_{E_a}E_i=0,
\]
and
\[
(\nabla_{E_i}E_j)^{\mathrm{rad}}
=
-\frac{F'}F\delta_{ij}\partial_s,
\qquad
(\nabla_{E_a}E_b)^{\mathrm{rad}}
=
-\frac{r'}r\delta_{ab}\partial_s.
\]
Together with the intrinsic curvature one of the round factors, these identities give \eqref{eq:fullcurvature}.

Suppose now that
\[
F,r>0,\qquad
F',r'\ge0,\qquad
r'<1,
\]
and
\begin{equation}\label{eq:northcriterion}
\frac{1-F'^2}{F^2}
>
\frac{8FF'r'}{r^3},
\qquad
-\frac{F''}{F}
>
\frac{4F^2(r'')_+}{r^3}.
\end{equation}
Using \eqref{eq:graphestimates} and \eqref{eq:areaestimate} in \eqref{eq:fullcurvature}, the coefficient of
\[
|\lambda Y-\mu X|^2
\]
is strictly positive, as is the coefficient of
\[
|X\wedge Y|^2,
\]
while the coefficient of the vertical-area term
\[
|U\wedge V|^2
\]
is nonnegative. If the first two quantities vanished simultaneously, the projections
\[
\lambda\partial_s+X,
\qquad
\mu\partial_s+Y
\]
would be linearly dependent. Since the star-horizontal space is a graph over the radial-angular tangent space, the original horizontal vectors would then be dependent. Hence every star-horizontal two-plane has positive sectional-curvature numerator. This extends \cite[Lemma~5.2]{HLY} to the present polar action.

\paragraph{Choice of the warping functions.}
We now use the profiles of \cite[\S5.2]{HLY}. Let
\[
\eta:\mathbb R\to[0,1]
\]
be smooth and nondecreasing, with
\[
\eta=0
\quad\text{on }\left(-\infty,\frac14\right],
\qquad
\eta=1
\quad\text{on }\left[\frac12,\infty\right),
\]
and set
\[
C_\eta=\sup|\eta'|.
\]
Choose
\begin{equation}\label{eq:delta}
0<\delta\le\frac{F_a}{2},
\qquad
\delta^3
\le
\frac{1}{128A_0F_a(1+C_\eta)}.
\end{equation}

Define \(F\) by
\[
F'
=
e^{-F^2/(2\delta^2)},
\qquad
F(0)=0,
\]
and choose \(\ell_N\) so that
\[
F(\ell_N)=F_a.
\]
Equivalently,
\[
\ell_N
=
\int_0^{F_a}
e^{z^2/(2\delta^2)}\,dz.
\]
Set
\[
d=r_aq_s
\]
and define
\[
r(s)
=
r_a
-
d\int_s^{\ell_N}
\eta(v/\delta)\,dv.
\]
For sufficiently small \(\varepsilon>0\), we may require
\[
q_s\ell_N\le\frac12,
\qquad
d<1.
\]
Indeed,
\[
q_s=O(\varepsilon),
\qquad
d=O(\varepsilon^{3/2}),
\]
while \(\ell_N\) is independent of \(\varepsilon\) once \(\delta\) has been fixed. Hence these inequalities hold for
\[
0<\varepsilon<\varepsilon_N
\]
for some \(\varepsilon_N>0\).

Since
\[
\ell_N\ge F_a\ge2\delta,
\]
one has \(\eta=1\) near \(s=\ell_N\). Moreover,
\[
\frac{r_a}{2}\le r\le r_a,
\qquad
r'=d\,\eta(s/\delta),
\qquad
r''=\frac d\delta\,\eta'(s/\delta),
\]
and
\[
r(\ell_N)=r_a,
\qquad
\frac{r'(\ell_N)}{r_a}=q_s.
\]
The estimate \(r\ge r_a/2\) gives
\[
\frac d{r^3}
\le
\frac{8q_s}{r_a^2}
=
4A_0F_a e^{-\varepsilon A_0|\cos a|}
\le
4A_0F_a.
\]

The defining equation for \(F\) gives
\[
-\frac{F''}{F}
=
\delta^{-2}e^{-F^2/\delta^2},
\qquad
\frac{1-F'^2}{F^2}
=
\frac{1-e^{-F^2/\delta^2}}{F^2}.
\]
On the support of \(r''\),
\[
s\le\frac\delta2
\]
and \(F\le s\), so
\[
-\frac{F''}{F}
\ge
e^{-1/4}\delta^{-2}.
\]
Also,
\[
\frac{4F^2r''}{r^3}
\le
4A_0F_aC_\eta\,\delta
\le
\frac{C_\eta}{32(1+C_\eta)}\delta^{-2}.
\]
This proves the second inequality in \eqref{eq:northcriterion}; away from the support of \(r''\), its right-hand side vanishes.

For \(x>0\), the power series for \(\sinh\) gives
\[
\frac{2\sinh(x^2/2)}{x^3}
\ge
\frac1x+\frac{x^3}{24}
\ge
\frac{4}{3\,8^{1/4}}
>
\frac12.
\]
Substituting
\[
x=\frac F\delta
\]
yields
\[
\frac{1-F'^2}{F^2}
>
\frac{FF'}{2\delta^3}
>
32A_0F_aFF'
\ge
\frac{8FF'r'}{r^3}.
\]
Thus the first inequality in \eqref{eq:northcriterion} also holds, and
\[
r'\le d<1.
\]
The metric therefore has positive sectional-curvature numerator on every star-horizontal two-plane away from the centre.

\paragraph{Smoothness at the centre and passage to the quotient.}
The function
\[
F\longmapsto
\int_0^F e^{z^2/(2\delta^2)}\,dz
\]
is smooth and odd with derivative one at the origin. Its inverse is therefore smooth and odd, with expansion
\[
F(s)
=
s-\frac{s^3}{6\delta^2}
+O(s^5).
\]
Moreover, \(r\) is constant near \(s=0\), since \(\eta(s/\delta)=0\) there. These are the standard smoothness conditions for the doubly warped metric \eqref{eq:north} at the centre.

At the centre the base point is fixed by \(\rho\), so the star orbit is the entire \(S^3\)-fibre and the star-horizontal tangent space is the base tangent space. The limiting sectional curvature of the base factor is
\[
\delta^{-2}>0.
\]
Thus the star-horizontal two-planes are positively curved at the centre as well.

At \(u=1\), the Gram matrix of the infinitesimal star vectors is
\[
F^2K_y^*K_y+r^2\operatorname{Id}
\ge
r^2\operatorname{Id}.
\]
Hence the star orbits are nonsingular even at points where the rank of \(K_y\) drops. The star action is free and isometric, and O'Neill's formula gives a metric \(g_N\) on
\[
D_N=P_N/S^3_\star
\]
with
\[
\operatorname{sec}_{g_N}>0
\]
throughout the disk.

At the boundary \(s=\ell_N\),
\[
F(\ell_N)=F_a,
\qquad
r(\ell_N)=r_a,
\qquad
\frac{r'(\ell_N)}{r_a}=q_s,
\qquad
\mu_N:=\frac{F'(\ell_N)}{F_a}>0.
\]
With \(X\) and \(U\) denoting the angular and fibre components in scaled orthonormal frames, the outward unit normal is \(\partial_s\), and the second fundamental form of the source boundary is
\[
\mathcal B_N(X+U,X+U)
=
\mu_N|X|^2+q_s|U|^2.
\]
Thus the source boundary metric is
\[
F_a^2h_{S^{n-1}}+r_a^2h_{S^3},
\]
the same metric obtained from the southern filling.

\subsection{Boundary compatibility}

The boundary comparison follows the calculation of He--Liu--Yau \cite[\S6]{HLY}. In the common northern trivialisation, the source boundary metrics of the two fillings are both
\[
F_a^2h_{S^{n-1}}+r_a^2h_{S^3},
\]
and the star actions are identical. Their quotient boundary metrics therefore agree under the identification defined by this common trivialisation. By Lemma~\ref{lem:attaching}, this identification is the map
\[
\sigma\colon\partial D_N\longrightarrow\partial D_S
\]
in the disk coordinates extending over \(D_N\) and \(D_S\). Hence
\[
h_N=\sigma^*h_S=:h.
\]

The quotient metric can also be written explicitly. In the common boundary marking, the quotient map is
\[
\Psi_\partial(x,u)=\rho(u)^{-1}x.
\]
At \(u=1\), write \(X\in T_yS(V)\) and \(U\in\operatorname{Im}\mathbb H\) for angular and fibre components in scaled orthonormal frames. Its differential is
\[
L(X,U)
=
F_a^{-1}X-r_a^{-1}K_yU.
\]
Since the source metric is Euclidean in these scaled components, the quotient cometric is
\[
h^{-1}
=
LL^*
=
F_a^{-2}h_{S^{n-1}}^{-1}
+
r_a^{-2}K_yK_y^*.
\]

Let \(B_N\) and \(B_S\) denote the outward second fundamental forms of the quotient boundaries. The outward unit normals to the source boundaries are
\[
\partial_s
\qquad\text{and}\qquad
-\partial_t
\]
on the northern and southern sides, respectively. Both are star-horizontal and project to the corresponding quotient normals. For a quotient boundary vector \(Y\), let
\[
X+U
\]
be its common star-horizontal lift in the boundary source metric. The second fundamental form of a quotient boundary is the source second fundamental form restricted to star-horizontal lifts. Therefore
\[
B_N(Y,Y)
=
\mu_N|X|^2+q_s|U|^2,
\]
while
\[
(\sigma^*B_S)(Y,Y)
=
-\mu_S|X|^2-q_s|U|^2.
\]
The fibre terms cancel, giving
\begin{equation}\label{eq:shapesum}
(B_N+\sigma^*B_S)(Y,Y)
=
(\mu_N-\mu_S)|X|^2.
\end{equation}
Since
\[
\mu_N>0>\mu_S,
\]
this is positive for every nonzero \(Y\). Indeed, the boundary star-horizontal space is the graph
\[
U=-\frac{F_a}{r_a}K_y^*X,
\]
so \(X=0\) would imply \(U=0\) and hence \(Y=0\).

For the quantitative bound required in the gluing argument, let \(b_j\) be the eigenvalues of
\[
K_yK_y^*.
\]
The hypothesis
\[
\|K_y\|_{\mathrm{op}}\le2
\]
gives
\[
0\le b_j\le4.
\]
If \(X\) is an eigenvector corresponding to \(b_j\), then
\[
|Y|_h^2
=
|X|^2+|U|^2
=
\left(
1+\frac{F_a^2b_j}{r_a^2}
\right)|X|^2.
\]
It follows from \eqref{eq:shapesum} that the corresponding eigenvalue of
\[
B_N+\sigma^*B_S
\]
relative to \(h\) is
\[
\frac{r_a^2}{r_a^2+F_a^2b_j}
(\mu_N-\mu_S).
\]
Since \(b_j\le4\),
\[
B_N+\sigma^*B_S
\ge
c_Bh,
\]
where
\[
c_B
=
\frac{r_a^2}{r_a^2+4F_a^2}
(\mu_N-\mu_S)
>0.
\]
Thus
\begin{equation}\label{eq:compat}
h_N=\sigma^*h_S,
\qquad
B_N+\sigma^*B_S\ge c_Bh.
\end{equation}

\subsection{Gluing}

We use the positive-sectional-curvature case of the Reiser--Wraith gluing theorem.

\begin{theorem}[\protect{Reiser--Wraith \cite[Theorem~A(i)]{RW}}]
\label{prop:gluing}
Let \((M_N,g_N)\) and \((M_S,g_S)\) be smooth Riemannian manifolds of the same dimension with compact boundary and strictly positive sectional curvature. Suppose that
\[
\phi\colon\partial M_N\longrightarrow\partial M_S
\]
is an isometry of the induced boundary metrics and that
\[
B_N+\phi^*B_S\ge0,
\]
where the second fundamental forms are taken with respect to the outward unit normals. Then
\[
M_N\cup_\phi M_S
\]
admits a smooth Riemannian metric with strictly positive sectional curvature. The metric may be chosen to agree with \(g_N\) and \(g_S\) outside an arbitrarily small neighbourhood of the gluing hypersurface.
\end{theorem}

For \(k=1\), the curvature condition \(Ric_k>0\) in \cite[Theorem~A(i)]{RW} is positive sectional curvature. The second fundamental form convention in \cite{RW} agrees with ours: Reiser--Wraith use the inward normal with the opposite sign, so that the boundary of a round ball is positive definite.

Applying Theorem~\ref{prop:gluing} with
\[
\phi=\sigma
\]
and using \eqref{eq:compat} gives a smooth metric of strictly positive sectional curvature on
\[
D_N\cup_\sigma D_S
\cong
P_\theta/S^3_\star.
\]

\subsection{Proof of Theorem~\ref{thm:HLYgeneral}}
The preceding subsections construct the southern and northern metrics, establish their positive sectional curvature, and verify the boundary conditions required by the Reiser--Wraith gluing theorem. To complete the proof, it remains to choose the parameters in a compatible order and apply the gluing theorem to the two quotient disks.
\begin{proof}
Choose the parameters in the following order. First fix
\[
\frac{\pi}{2}<a<b_0<\pi
\]
and the cutoff function \(\chi\). For the fixed polar bundle \(P_\theta\), the resulting connection on the southern piece has finite curvature bounds \(M_0\) and \(M_1\). Choose \(\Lambda\) satisfying \eqref{eq:Lambda}, and then choose \(A_0>0\) so that
\[
A_0|\cos a|\ge\Lambda.
\]
Next choose the cutoff \(\eta\) and \(\delta>0\) satisfying \eqref{eq:delta}. This determines the northern warping function \(F\) and the length
\[
\ell_N
=
\int_0^{F_a}e^{z^2/(2\delta^2)}\,dz.
\]

All of these choices are independent of \(\varepsilon\). The southern construction gives \(\varepsilon_S>0\) such that \(g_S\) has strictly positive sectional curvature whenever
\[
0<\varepsilon<\varepsilon_S.
\]
On the northern side,
\[
q_s=\frac12\varepsilon A_0F_a
\]
and
\[
d=r_aq_s
=
O(\varepsilon^{3/2})
\qquad
(\varepsilon\to0).
\]
Since \(\ell_N\) is fixed, the conditions
\[
q_s\ell_N\le\frac12,
\qquad
d<1
\]
hold for all sufficiently small \(\varepsilon>0\). Thus there exists \(\varepsilon_N>0\) such that the northern construction gives a metric \(g_N\) of strictly positive sectional curvature for
\[
0<\varepsilon<\varepsilon_N.
\]

Choose
\[
0<\varepsilon<\min\{\varepsilon_S,\varepsilon_N\}.
\]
Then \((D_N,g_N)\) and \((D_S,g_S)\) have strictly positive sectional curvature, including along their boundaries. By \eqref{eq:compat},
\[
h_N=\sigma^*h_S,
\qquad
B_N+\sigma^*B_S\ge c_Bh>0.
\]
The Reiser--Wraith gluing theorem, Theorem~\ref{prop:gluing}, therefore gives a smooth metric of strictly positive sectional curvature on
\[
D_N\cup_\sigma D_S.
\]
By Lemma~\ref{lem:attaching},
\[
D_N\cup_\sigma D_S
\cong
P_\theta/S^3_\star.
\]
Hence \(P_\theta/S^3_\star\) admits a smooth Riemannian metric with strictly positive sectional curvature.
\end{proof}

\begin{remark}[The representation bound can be removed]
\label{rem:general_bound}
The hypothesis
\[
\|K_y\|_{\mathrm{op}}\le2
\]
in Theorem~\ref{thm:HLYgeneral} is a hypothesis of the particular He--Liu--Yau construction used in this paper. It is not necessary for the existence of a positively curved metric on a star quotient.

Indeed, let \(E\to S(\mathbb Re\oplus V)\) be any star bundle, with \(V\) an arbitrary finite-dimensional orthogonal \(S^3\)-representation. By Proposition~\ref{prop:polar}, \(E\) is equivariantly isomorphic to a polar bundle \(P_\theta\) for a smooth conjugation-equivariant map
\[
\theta:S(V)\longrightarrow S^3.
\]
By Lemma~\ref{lem:attaching},
\[
E/S^3_\star
\cong
D(V)_N\cup_{\sigma_{\rho,\theta}}D(V)_S,
\qquad
\sigma_{\rho,\theta}(x)=\rho(\theta(x))^{-1}x.
\]
Set \(\beta=\theta^{-1}\). Then \(\beta\) is again conjugation-equivariant and
\[
\sigma_{\rho,\theta}(x)
=
\rho(\beta(x))x
=
J_\beta(x).
\]
Deng--Hu--Zhang \cite[Proposition~5.1]{DHZ}, together with the Reiser--Wraith gluing theorem \cite[Theorem~A(i)]{RW}, therefore implies that \(E/S^3_\star\) admits a smooth metric of positive sectional curvature, with no bound on the infinitesimal action fields.
\end{remark}

\section{Homotopy spheres from star bundles}\label{sec:speranca}

We now apply the preceding results to the special \(S^3\)-\(S^3\) bundles constructed by Sperança \cite{Sp}. His examples are obtained by equivariant pullback of the Gromoll--Meyer bundle
\[
\pi_{GM}\colon Sp(2)\longrightarrow S^7
\]
and carry a second free \(S^3\)-action commuting with the principal action. We will show that these are star bundles whose infinitesimal base actions satisfy
\[
\|K_y\|_{\mathrm{op}}\le2.
\]
Proposition~\ref{prop:polar} then puts them in polar form, and Theorem~\ref{thm:HLYgeneral} applies.

We use right principal actions, as in the preceding sections. Sperança writes the principal action on the left; we pass to the equivalent right action
\[
ph=h^{-1}\mathbin\bullet p.
\]

\subsection{The two star bundles}

We first identify the two bundles that will enter the curvature construction.

\subsubsection*{Dimension eight.}
Sperança obtains a special \(S^3\)-\(S^3\) bundle
\[
\pi_{11}\colon E^{11}\longrightarrow S^8
\]
by an equivariant pullback of the Gromoll--Meyer bundle; see
\cite[Remark~2.1 and (2.5)--(2.7)]{Sp}. Using coordinates
\[
S^8=S(\mathbb R\oplus\mathbb H\oplus\mathbb H),
\qquad
(\lambda,x,w)\in S^8,
\]
the induced action on the base is
\begin{equation}\label{eq:rho8}
\rho_8(q)(\lambda,x,w)
=
(\lambda,qx,qwq^{-1}).
\end{equation}
The vector
\[
e_8=(1,0,0)
\]
is fixed by this action, and its orthogonal complement is
\[
V_8=\mathbb H\oplus\mathbb H.
\]
Thus
\[
\mathbb R\oplus\mathbb H\oplus\mathbb H
=
\mathbb Re_8\oplus V_8,
\]
with \(\rho_8\) orthogonal and fixing the distinguished axis \(\mathbb Re_8\). Together with Sperança's special \(S^3\)-\(S^3\) structure, this makes
\[
E^{11}\longrightarrow S^8
\]
a star bundle.

\subsubsection*{Dimension ten.}
The ten-dimensional construction is similar. Sperança defines a special \(S^3\)-\(S^3\) bundle
\[
\pi_{13}\colon E^{13}\longrightarrow S^{10}
\]
by another equivariant pullback of the Gromoll--Meyer bundle; see
\cite[(2.8)--(2.10)]{Sp}. In the presentation
\[
S^{10}
=
S(\operatorname{Im}\mathbb H\oplus\mathbb H\oplus\mathbb H),
\qquad
(p,w,x)\in S^{10},
\]
the base action is
\begin{equation}\label{eq:rho10}
\rho_{10}(q)(p,w,x)
=
(p,qw,qxq^{-1}).
\end{equation}
Here the fixed vector
\[
e_{10}=(0,0,1)
\]
gives the orthogonal decomposition
\[
\operatorname{Im}\mathbb H\oplus\mathbb H\oplus\mathbb H
=
\mathbb Re_{10}\oplus V_{10},
\]
where
\[
V_{10}
=
\operatorname{Im}\mathbb H
\oplus\mathbb H
\oplus\operatorname{Im}\mathbb H.
\]
Again the representation is orthogonal and fixes the distinguished axis, so
\[
E^{13}\longrightarrow S^{10}
\]
is a star bundle.

Sperança determines the diffeomorphism types of the resulting star quotients.

\begin{theorem}[\protect{Sperança \cite[Theorem~1]{Sp}}]
\label{thm:speranca}
The quotient
\[
E^{11}/S^3_\star
\]
is the exotic \(8\)-sphere. The quotient
\[
E^{13}/S^3_\star
\]
represents a generator of the order-three subgroup of \(\Theta_{10}\) consisting of homotopy \(10\)-spheres that bound spin manifolds.
\end{theorem}

It remains to verify the geometric hypothesis of Theorem~\ref{thm:HLYgeneral}. For both representations this reduces to the same estimate for the infinitesimal action fields.

\subsection{The action-field bound}

For \(y\in S(V)\), let
\[
K_y\colon\operatorname{Im}\mathbb H\longrightarrow T_yS(V)
\]
be the infinitesimal action map
\[
K_y\xi
=
\left.\frac{d}{d\tau}\right|_{\tau=0}
\rho(e^{\tau\xi})y.
\]
The norms are taken with respect to \(Q\) on \(\operatorname{Im}\mathbb H\) and the unit round metric on \(S(V)\), as in \cite[\S1.3]{HLY}.

\begin{lemma}\label{lem:K}
For both \(\rho_8\) and \(\rho_{10}\),
\[
\|K_y\|_{\mathrm{op}}\le2
\]
for every \(y\in S(V)\).
\end{lemma}

\begin{proof}
For \(\xi,v\in\mathbb H\) with \(\xi\in\operatorname{Im}\mathbb H\),
\[
\xi v-v\xi
=
[\xi,\operatorname{Im}v]
=
2\,\xi\times\operatorname{Im}v.
\]

In dimension eight, write
\[
y=(x,w)\in S(\mathbb H\oplus\mathbb H).
\]
Differentiating \eqref{eq:rho8} gives
\[
K_y\xi
=
(\xi x,[\xi,w])
=
(\xi x,2\xi\times\operatorname{Im}w).
\]
Consequently,
\[
\begin{aligned}
|K_y\xi|^2
&=
|\xi|^2|x|^2
+
4|\xi\times\operatorname{Im}w|^2\\
&\le
|\xi|^2\bigl(|x|^2+4|\operatorname{Im}w|^2\bigr)\\
&\le
4|\xi|^2,
\end{aligned}
\]
because
\[
|x|^2+|w|^2=1.
\]

In dimension ten, write
\[
y=(p,w,x)
\in
S(\operatorname{Im}\mathbb H\oplus\mathbb H\oplus\operatorname{Im}\mathbb H).
\]
Differentiating \eqref{eq:rho10} gives
\[
K_y\xi
=
(0,\xi w,[\xi,x])
=
(0,\xi w,2\xi\times x).
\]
Hence
\[
\begin{aligned}
|K_y\xi|^2
&=
|\xi|^2|w|^2
+
4|\xi\times x|^2\\
&\le
|\xi|^2\bigl(|w|^2+4|x|^2\bigr)\\
&\le
4|\xi|^2,
\end{aligned}
\]
since
\[
|p|^2+|w|^2+|x|^2=1.
\]
Thus \(\|K_y\|_{\mathrm{op}}\le2\) in both cases.
\end{proof}

The action-field estimate supplies the remaining hypothesis of Theorem~\ref{thm:HLYgeneral}. We can therefore combine the polar normal form, the geometric construction of Section~\ref{sec:HLY}, and Sperança's identification of the quotients to prove the main results.

\subsection{Proofs of the main results}

We can now prove Theorems~A and~B by combining the polar normal form with the action-field estimate.

\begin{proof}[Proof of Theorem~A]
The bundle
\[
E^{11}\longrightarrow S^8
\]
is a star bundle for the representation \(\rho_8\). By Proposition~\ref{prop:polar}, it is equivariantly isomorphic to a polar bundle \(P_\theta\), and Lemma~\ref{lem:K} gives
\[
\|K_y\|_{\mathrm{op}}\le2
\]
on its equator. Theorem~\ref{thm:HLYgeneral} therefore gives a metric of strictly positive sectional curvature on
\[
E^{11}/S^3_\star.
\]
By Theorem~\ref{thm:speranca}, this quotient is the exotic \(8\)-sphere.
\end{proof}

\begin{proof}[Proof of Theorem~B]
The same argument applies to the star bundle
\[
E^{13}\longrightarrow S^{10}
\]
with base representation \(\rho_{10}\). Proposition~\ref{prop:polar}, Lemma~\ref{lem:K}, and Theorem~\ref{thm:HLYgeneral} give a metric of strictly positive sectional curvature on
\[
E^{13}/S^3_\star.
\]
By Theorem~\ref{thm:speranca}, this quotient represents a generator of the order-three subgroup of \(\Theta_{10}\). Reversing orientation gives the other nontrivial element of that subgroup and leaves the Riemannian metric on the underlying smooth manifold unchanged.
\end{proof}

\begin{proof}[Proof of Corollary~C]
The standard \(8\)-sphere carries its round metric, while Theorem~A gives a metric of strictly positive sectional curvature on the exotic \(8\)-sphere. Hence every homotopy \(8\)-sphere admits a metric of strictly positive sectional curvature.

For dimension ten, every homotopy \(10\)-sphere is spin, since
\[
H^2(\Sigma^{10};\mathbb Z/2)=0
\]
and therefore \(w_2(\Sigma^{10})=0\). The spin structure is unique because
\[
H^1(\Sigma^{10};\mathbb Z/2)=0.
\]
The real Dirac index defines an additive homomorphism
\[
\alpha\colon\Theta_{10}\longrightarrow KO_{10}\cong\mathbb Z/2;
\]
see \cite[pp.~41--42]{Hi}. This homomorphism is nonzero in dimension ten \cite[p.~44]{Hi}. Since
\[
\Theta_{10}\cong\mathbb Z/6,
\]
its kernel is therefore the unique subgroup of order three.

If a homotopy \(10\)-sphere admits a metric of positive scalar curvature, then its \(\alpha\)-invariant vanishes by the spin Dirac index obstruction \cite{Hi}. Conversely, Theorem~B gives metrics of strictly positive sectional curvature on the two nontrivial elements of \(\ker\alpha\), while the trivial element is represented by the round sphere. Hence every element of \(\ker\alpha\) admits a metric of strictly positive sectional curvature. Since positive sectional curvature implies positive scalar curvature, a homotopy \(10\)-sphere admits a metric of positive sectional curvature if and only if it admits a metric of positive scalar curvature.
\end{proof}

\end{document}